\documentclass[11pt]{amsart}

\usepackage{amsmath}
\usepackage{amssymb,amsfonts,mathrsfs, amsthm, stmaryrd}
\usepackage[colorlinks=true,linkcolor=blue,citecolor=blue]{hyperref}
\usepackage[margin=2 cm,heightrounded=true,centering]{geometry}

\usepackage{amscd}
\usepackage{verbatim}
\usepackage{euscript}

\newcommand{\dtime}{\frac{d}{d t}} 
\newcommand{\Lie}[2]{\mathcal{L}_{#1} #2} 
\newcommand{\dvar}[1]{\partial_{#1}} 

\newcommand{\amb}{\mathcal{O}} 
\newcommand{\sP}{\mathcal{P}} 

\newcommand{\gbar}{\overline{g}}
\newcommand{\inn}[2]{\left\langle #1, #2 \right\rangle}

\newcommand{\dvh}{dv^h} 

\newcommand{\sT}{\mathcal{T}}

\newcommand{\sO}{\mathcal{O}}

\renewcommand{\P}{\mathcal{P}}        

\newcommand{\bR}{\mathbb{R}}

\newcommand{\id}{\mathrm{id}}

\newcommand{\tr}{\mathrm{tr}}       
\newcommand{\snorm}[2]{||#1||_{#2}} 
\newcommand{\vep}{\varepsilon}

\newcommand{\overarrow}[1]{\xrightarrow{ \hspace{0.5cm} #1 \hspace{0.5cm}}}

\newtheorem{main}{Theorem}

\newtheorem{thm}{Theorem}[section]

\newtheorem{lemma}[thm]{Lemma}
\newtheorem{prop}[thm]{Proposition}



\begin{document}
\title[Geometric flows]{Uniqueness for Some Higher-Order Geometric Flows}
\author{Eric Bahuaud}
\address{Department of Mathematics,
Seattle University.}
\email{bahuaude(at)seattleu.edu}

\author{Dylan Helliwell}
\address{Department of Mathematics, 
Seattle University.}
\email{helliwed(at)seattleu.edu}

\date{\today}

\begin{abstract}
We show that solutions to certain higher-order intrinsic geometric flows on a compact manifold, including some flows generated by the ambient obstruction tensor, are unique.  With the goal of providing a complete self-contained proof, details surrounding map covariant derivatives and a careful application of the DeTurck trick are provided.
\end{abstract}

\maketitle


\section{Introduction}

An intrinsic curvature flow on a compact smooth manifold is a time-dependent Riemannian metric that satisfies a system of the form
\begin{equation}
\label{intro-flow}
\left\{ \begin{aligned} 
\dvar{t} g &= T^{g} \\
g(0) &= h
\end{aligned} \right.
\end{equation}
where $T^g$ is a natural tensor constructed from contractions of curvature (see equation \eqref{ansatz} below and the following section for a precise statement).  The Ricci flow is undoubtably the most studied example, but there has also been recent interest in higher-order geometric flows based on the negative gradients of quadratic curvature functionals including the Bach tensor and more generally the ambient obstruction tensors.

The purpose of this note is to elucidate the uniqueness proof carefully in the higher-order setting.  While uniqueness of certain fourth-order gradient flows of quadratic curvature functionals was addressed in \cite{Bour} and \cite{Streets}, we feel the treatment of uniqueness in those papers was somewhat terse and it would be worthwhile to have a complete self-contained proof available in the literature.  Moreover our work here is applied to various other higher-order flows, including those by the ambient obstruction tensor that we introduced in \cite{BahuaudHelliwell}.

When working with \eqref{intro-flow} the diffeomorphism invariance of $T^g$ complicates the proof of local existence and uniqueness as the system \eqref{intro-flow} is not strictly parabolic.  The standard argument for resolving this problem is the so-called ``DeTurck trick,'' which we briefly explain.  For the proof of local existence, an analysis of the principal symbol of the linearization of $T^g$ leads to the definition of an appropriate time-dependent vector field $W^{g,h}$ requiring $g$ and also a background reference metric $h$ such that the adjusted system
\begin{equation}
\label{intro-adj-flow}
\left\{ \begin{aligned} 
\dvar{t} g &= T^{g} + \Lie{W^{g,h}}{g} \\
g(0) &= h
\end{aligned} \right.
\end{equation}
is strictly parabolic.  Now classical parabolic theory may be applied to obtain a solution for a short time.  Then one solves for a time-dependent diffeomorphism $\varphi: I \times M \to M$ that satisfies
\begin{equation}
\label{intro-vf-flow}
\left\{ \begin{aligned} 
\dvar{t} \varphi &= -W^{g,h} \circ \varphi \\
\varphi(0) &= \id.
\end{aligned} \right.
\end{equation}
Finally, it may be checked that the time-dependent metric $\gbar(t) = [\varphi(t)]^* g(t)$ satisfies \eqref{intro-flow} for a short time, completing the proof of local existence.  

The proof of uniqueness is more subtle and involves three major steps.  First, uniqueness for the adjusted flow \eqref{intro-adj-flow} is established.  For second order flows, the maximum principle is a key tool, but for higher-order flows, we instead derive integral energy estimates.  Second, the vector field $W^{g,h}$ is expressed in terms of the metric $\gbar$ that solves \eqref{intro-flow}.  For appropriate choices of $W^{g,h}$, this gives rise to a parabolic equation for the time-dependent diffeomorphism $\varphi$:
\begin{equation*}
\left\{ \begin{aligned} 
\dvar{t} \varphi &= E^{\gbar,h}(\varphi) \\
\varphi(0) &= \id.
\end{aligned} \right.
\end{equation*}
Existence of solutions to this equation is known.  Third, given two solutions $\gbar_1$ and $\gbar_2$ of \eqref{intro-flow} with matching initial conditions, we find time-dependent diffeomorphisms $\varphi_i$, $i = 1,2$, solving
\begin{equation*}
\left\{ \begin{aligned} 
\dvar{t} \varphi_i &= E^{\gbar_i,h}(\varphi_i) \\
\varphi_i(0) &= \id.
\end{aligned} \right.
\end{equation*}
If we then set $g_i(t) = [(\varphi_i(t)]_* \gbar_i(t)$, and let $-W^{g_i, h}$ be the infinitesimal generator of the flow of $\varphi_i$, so that 
\[
E^{\gbar_i,h} = -W^{g_i, h} \circ \varphi_i,
\]
then each $g_i$ satisfies the adjusted flow \eqref{intro-adj-flow}.  Knowing the uniqueness of the adjusted flow now allows us to conclude that $g_1 = g_2$.  Then by uniqueness for ODE's applied to \eqref{intro-vf-flow}, we find $\varphi_1 = \varphi_2$ and so $\gbar_1 = \gbar_2$ on their common interval of existence, completing the proof.

Note that in the case of Ricci flow $(W^{g,h})^{\gamma} = g^{\alpha \beta}\bigl[(\Gamma^g)_{\alpha \beta}^{\gamma} - (\Gamma^h)_{\alpha \beta}^{\gamma}\bigr]$, where $\Gamma^g$ and $\Gamma^h$ are the Christoffel symbols for $g$ and $h$ respectively, and $E^{\gbar,h} = \Delta^{\gbar,h}$, the harmonic map Laplacian.  In the higher-order cases, care must be exercised to obtain an elliptic operator.

The flows of interest in this paper stem from polynomial natural $\binom{2}{0}$-tensors where the top order terms are linear combinations of $\Delta^k Ric$, $\Delta^k S g$, and $\Delta^{k-1} \nabla^2 S$, where $Ric$ and $S$ are Ricci and scalar curvature for $g$.  The ambient obstruction tensor is of this form, as are the gradients of various curvature functionals.  More precisely, for $k\geq 1$ and for constants $a$, $b$, $c$ to be described, set
\begin{equation} \label{ansatz}
T^g_{k,a,b,c} = (-1)^{k+1} c\bigl(\Delta^{k} Ric
			+ a \Delta^{k} S g - b \Delta^{k-1} \nabla^2 S\bigr) + T_{2k+1},
\end{equation}
where $T_{2k+1}$ is a lower order term.  We may also consider the case $k=0$ if we require $b=0$.
Considering flows generated by these tensors, our main result is then the uniqueness part of the following:

\begin{main} \label{mainuniquetheorem}
Let $c > 0$, $b \in \bR$ and $a > -\frac{1}{2(n-1)}$.  Then, on any compact manifold of dimension $n$ and for any smooth initial metric $h$, there exists a unique smooth short-time solution to
\begin{equation}
\tag{$\ast$}\label{originalflow}
\left\{ \begin{aligned} 
\dvar{t} g &= T^g_{k,a,b,c} \\
g(0) &= h.
\end{aligned} \right.
\end{equation}
\end{main}

Existence was established for this type of flow in \cite{BahuaudHelliwell}.  Related work was done by \cite{Streets} on flow by the gradient of $\int_M |Rm|^2\, d v$ and by \cite{Bour} who considered \eqref{originalflow} with $k = 1$.

As described above, the proof of Theorem \ref{mainuniquetheorem} consists of three main steps.  The first step shows existence/uniqueness for strictly parabolic systems, in other words for flows adjusted by the DeTurck vector fields.  

\begin{main} \label{thm:main-adj-flow}
Let $T^g$ be a polynomial natural symmetric $\binom{2}{0}$-tensor for a single metric $g$ in dimension $n$ with the property that there is a polynomial natural vector field $W^{g,h}$ for two metrics $g$ and $h$ in dimension $n$, such that the linearization of $-(T^g+ \Lie{W^{g,h}}{g})$ at any metric $h$ is a strongly elliptic linear operator of order $2m$.  Then, on any compact manifold of dimension $n$ and for any smooth initial metric $h$, there exists a unique smooth short-time solution to
\begin{equation}
\tag{$\ast \ast$}\label{originalflow-adj}
\left\{ \begin{aligned} 
\dvar{t} g &= T^g+ \Lie{W^{g,h}}{g} \\
g(0) &= h.
\end{aligned} \right.
\end{equation}
\end{main}

Note that uniqueness of strictly parabolic systems can be derived using energy methods, and was recently proved in \cite{MM} for systems with a `product type' top order term.  For completeness we give a slight adaptation of this argument. 

The second step in the proof of Theorem \ref{mainuniquetheorem} shows that an appropriate elliptic operator for the diffeomorphism may be found to reduce the uniqueness question for the pure geometric flow \eqref{originalflow} to the uniqueness of the adjusted flow in \eqref{originalflow-adj}.  It was some of the technical difficulty in this step that led us to write this paper, and we indicate some of these technicalities along the way.

The third step in the proof of Theorem \ref{mainuniquetheorem} combines uniqueness for the adjusted flow with existence of solutions to the parabolic problem arising from adjusting the flow to attain uniqueness for the pure flow.  This step does not have any significant challenges, but is included for completeness.

A special case of interest is flow by the ambient obstruction tensor which is described in more detail in Section \ref{ambientsection} .  An immediate application of Theorem \ref{mainuniquetheorem} is the following:

\begin{main} \label{ambientflowtheorem}
On any compact manifold of even dimension $n \geq 4$, for all $\alpha > 0$ and any smooth initial metric $h$, there is a unique smooth short-time solution to the following flow:
\begin{equation*}
\left\{ \begin{aligned} 
\dvar{t} g &= \sO_n + (-1)^{\frac{n}{2}} \alpha (\Delta^{\frac{n}{2}-1} S) g \\
g(0) &= h.
\end{aligned} \right.
\end{equation*}
where $\sO_n$ and $S$ are the ambient obstruction tensor and scalar curvature for $g$.
\end{main}

The remainder of this paper is organized as follows.  In Section \ref{background} we review background material including elliptic operators, G{\aa}rding's inequality, polynomial natural tensors and the definition of the ambient obstruction tensor.  We prove the uniqueness part of Theorem \ref{thm:main-adj-flow} in Section \ref{uniqueness:adj-flow}.  In Section \ref{mcd} we describe the map covariant derivative.  Finally in Section \ref{uniqueness:pure-flow}, we complete the proof of Theorem \ref{mainuniquetheorem}.

We are happy to thank Jack Lee for helpful discussions during the course of this work.


\section{Background} \label{background}

Let $M$ be a smooth compact manifold of dimension $n$.  Let $h$ be a given smooth metric.  This will serve as a background metric and will also be the initial condition for the flow.  We will also consider the space of square-integrable tensors $F$ with respect to this metric, denoted by $L^2(M,h)$ and equipped with the norm
\[ || F ||_{L^2} = \left( \int_M (|F|^h)^2\, \dvh \right)^{\frac{1}{2}}. \]
We refer the reader to Appendix A of \cite{Besse} for the definitions of the Sobolev spaces $L^2_k$ (the set of tensors with $k$ $h$-covariant derivatives in $L^2$), and the general setting of linear differential operators between vector bundles. 

\subsection{Strong ellipticity and G{\aa}rding's inequality}

Letting $\Sigma^2(M)$ be the bundle of symmetric $\binom{2}{0}$-tensors, consider a linear differential operator $P: C^{\infty}( \Sigma^2(M)) \longrightarrow C^{\infty}( \Sigma^2(M) )$ of order $2m$ with possibly time-dependent coefficients given in local coordinates $\{x^1, \ldots, x^n\}$ by
\[
P = \sum_{|\beta| \leq 2m} a_{\beta} D^{\beta},
\]
where $\beta = ( \beta_1, \cdots, \beta_n )$, $D^{\beta} = \partial_{x^1}^{\beta_1} \cdots \partial_{x^n}^{\beta_n}$, and $a_{\beta} = (a_{\beta}(x,t))_K^J$, where $J$ and $K$ range from 1 to $N = \dim(\Sigma^2(M))$.  For $x \in M$, fixed $t$ and $\xi \in T^*_x M$, the principal symbol of $P$, $\sigma_{\xi}(P): \Sigma^2(M)_x \longrightarrow \Sigma^2(M)_x$ is given by 
\[
\sigma_{\xi}(P)^J_K = (-1)^m \sum_{|\beta| = 2m} ( a_{\beta}(x,t))^J_K \xi^{\beta},
\]
and we say $P$ is strongly elliptic if there exists a time-independent constant $\Lambda > 0$ such that
\[
(-1)^m \sum_{|\beta| = 2m} ( a_{\beta}(x,t))^J_K \xi^{\beta} \eta_J \eta^K \geq \Lambda |\xi|^{2m} |\eta|^2,
\]
for all $x$, $t$, $\xi$ and $\eta$.

One of the key facts about strongly elliptic operators that we will need is the following version of G{\aa}rding's inequality, which we will use for time-independent operators $P$.

\begin{thm} \label{Garding}
Let $P: C^{\infty}( \Sigma^2(M)) \longrightarrow C^{\infty}( \Sigma^2(M) )$ be a strongly elliptic linear operator of order $2m$ on a compact Riemannian manifold $(M, h)$.  Then there are constants $c > 0$ and $K$ such that
\[
\int_M \inn{Pu}{u} \, \dvh \geq c \snorm{u}{L^{2}_m}^2 - K \snorm{u}{L^2}^2.
\]
\end{thm}

For the proof of G{\aa}rding's inequality for strongly elliptic systems in $\bR^n$, we refer the reader to \cite{Giaquinta}.  The result for vector bundles follows by standard techniques using partitions of unity.  A lucid presentation of some of these local-to-global techniques is also given in the classical text by \cite{Schechter}.  
 
\subsection{Polynomial Natural Tensor Structure}

A tensor $T~=~T^{g_1,~\cdots,~g_m}$ is a polynomial natural tensor of order $k$ for $m$ metrics in dimension $n$ if for any $n$-manifold $M$ and any diffeomorphism $f$ of $M$, we have $T^{f^*g_1, \cdots, f^*g_m} = f^*\bigl(T^{g_1, \cdots, g_m}\bigr)$, and if any any coordinate system its coefficients are polynomials in the coefficients of coordinate derivatives of $g_i$ up to order $k$, and $g_i^{-1}$.  We refer the reader to our earlier paper \cite{BahuaudHelliwell} for further details.

In order to represent terms where the particular structure is not important, let $A \star B$ be a linear combination of contractions of $A \otimes B$.  Furthermore, let $\P(\partial^{p_1} s_1, \partial^{p_2} s_2, \ldots)$ denote an expression with components that are polynomials consisting of contractions in the components of $s_r$ and their coordinate derivatives up to order $p_r$ when represented locally. 

\subsection{The Ambient Obstruction Tensor} \label{ambientsection}

Let $n \geq 4$ be even.  The ambient obstruction tensor arises in the study of conformally compact Einstein metrics and plays an important role in conformal geometry, see \cite{FeffermanGraham1, FeffermanGraham2} and the references therein for examples.  Our interest here is primarily in that the ambient obstruction tensor,\label{amot} $\sO_n$, is a conformally invariant polynomial natural tensor of order $n$ for a metric $g$ in dimension $n$.  This tensor is symmetric, trace free, and divergence free.  For our purposes we will only need the explicit structure of the top order terms.  Let $c_n = 1/[(-2)^{\frac{n}{2}-2}(\frac{n}{2}-2)!]$. Then
\[ \sO_n = c_n \left( \Delta^{\frac{n}{2} - 1} P - \frac{1}{2(n-1)} \Delta^{\frac{n}{2}-2} \nabla^2 S \right) + T_{n-1},\]
where 
\[ P = \frac{1}{n-2} \left( Ric - \frac{1}{2(n-1)} S g \right), \]
and $Ric, S$ are the Ricci and scalar curvatures of $g$, $P$ is the Schouten tensor, and $T_{n-1}$ is a polynomial natural tensor of order $n-1$, the specific structure of which is not important for our analysis.  Note that the ambient obstruction tensor $\amb_n$ is an example of a tensor of the form of equation \eqref{ansatz} with coefficients
\[
k = \frac{n}{2}-1,
a = \frac{-1}{2(n-1)},
b = \frac{(n-2)}{2(n-1)},
c = \frac{1}{(n-2)(2)^{\frac{n}{2}-2} \left(\frac{n}{2} - 2\right)!}
\]

Note that $a$ is at the critical value where Theorem \ref{mainuniquetheorem} fails, thus necessitating the condition that $\alpha > 0$ in Theorem \ref{ambientflowtheorem}.  As a special case, when $n = 4$, we find for the correct choice of $T_{3}$, $T^g_{k,a,b,c}$ is the Bach tensor $B$, ($B = \sO_4$) and we have
\begin{align*}
k = 1,
a = -\frac{1}{6},
b = \frac{1}{3},
c = \frac{1}{2}.
\end{align*}


\section{Uniqueness for the Adjusted Flow} \label{uniqueness:adj-flow}

This section outlines the proof of the uniqueness part of Theorem \ref{thm:main-adj-flow}.  (The proof of existence may be found for example in \cite{BahuaudHelliwell} or \cite{MM}.)  Our argument here is essentially that given in \cite{MM}, but included here for completeness.  After setting up the problem in the appropriate way, we state a technical lemma that allows us to bound certain quadratic terms that arise.  We then construct a differential inequality for a combination of norms of the difference between two candidate solutions and use it to prove uniqueness for the adjusted flow.

\subsection{Setup and estimates}

Suppose that $\sT^{g,h}$ is a polynomial natural tensor that satisfies the hypothesis of Theorem \ref{thm:main-adj-flow}, where the linearization of $\sT^{g,h} = -( T^g + \Lie{W^{g,h}}{g} )$ is strongly elliptic.  Suppose that $g_1(t) = h+u(t)$ and $g_2(t) = h+v(t)$ are two solutions to  \eqref{originalflow-adj} with the same initial metric $h$.  Then following the work done in our earlier paper \cite[Section 4.1]{BahuaudHelliwell}, we may write
\[
\dvar{t} u = \sT^{h+u,h} = I^h + L^h(u) + Q^h(u)
\]
and
\[
\dvar{t} v = \sT^{h+v,h} = I^h + L^h(v) + Q^h(v).
\]
where $I^h = \sT^{h,h}$ is the nonhomogenous term, $L^h(u) = \frac{d}{ds}|_{s=0} \sT^{h+ u(s),h}$ is the linearization at $h$, and $Q^h(u)$ represents a linear combination of ``quadratic'' terms of the form
\begin{equation*} \label{structure-of-Q}
h^{\alpha} \star h^{-\beta} \star (Rm^h)^{\gamma} \star (h+u)^{-\delta}
	\star (\nabla^h)^{j_1} u \star \cdots \star (\nabla^h)^{j_l} u,
\end{equation*}
where $l \geq 2$, $0 \leq j_l \leq 2m$ and $h^{-\beta}$ and $(h+u)^{-\delta}$ represents $\beta$ and $\delta$ copies of the inverses of $h$ and $h+u$ respectively.

With these adjustments made, we note that all geometric objects in this section are built with respect to $h$ and as such, for this section only, unadorned geometric objects are to be interpreted as depending on $h$.

Note that $L = L^h$ is now a polynomial natural time-independent differential operator in the metric $h$, and this fact allows us to use energy methods with respect to $L^2$ spaces defined in terms of $h$.  By hypothesis, $-L$ is strongly elliptic.

Set $w = g_1 - g_2 = u-v$, and observe
\begin{equation} \label{differenceequation}
\dvar{t}w = L(w) + Q(u) - Q(v).
\end{equation}
Let $\nabla$ be the Levi-Civita connection for $h$ and note that
\begin{align*}
\dvar{t} (\nabla^k w) &= \nabla^k (\dvar{t} w) \\
				&= \nabla^k \left(L(w) + Q(u) - Q(v) \right) \\
				&= \hat{L}(\nabla^k w) + \nabla^k\left({Q}(u) - {Q}(v)\right),
\end{align*}
where $\hat{L}$ and $L$ have the same structure at top order and differ by lower order terms involving the curvature of $h$.  

The next Lemma will allow us to estimate the quadratic terms.  The proof follows from the explicit local Lipschitz structure of $Q$ as described in \cite{BahuaudHelliwell} and interpolation inequalities.  The constants that appear depend on the precise algebraic structure of $Q$.

\begin{lemma} \label{quadest}
Let $u$ and $v$ be smooth symmetric two-tensors which vanish at $t = 0$ and let $w = u-v$.  Then there exists a constant $C>0$ such that \[
\snorm{Q(u) - Q(v)}{L^2} \leq C \snorm{w}{L^2_{2m}}.
\]
Moreover, for any $\varepsilon > 0$, there is a constant $C_{\varepsilon} > 0$ such that
\[
\snorm{Q(u) - Q(v)}{L^2} \leq C_{\vep} \snorm{w}{L^2} + \vep \snorm{\nabla^{2m} w}{L^2}.
\]

\end{lemma}

\subsection{A differential inequality}

We now derive a differential inequality for a particular combination of norms of $w$.

\begin{prop} \label{differentialinequalityprop}
Let $w$ solve \eqref{differenceequation}.  Then there is a constant $K$ such that
\begin{equation*}
\dtime \left(\snorm{w}{L^2}^2 + \snorm{\nabla^m w}{L^2}^2\right)
		\leq K \left(\snorm{w}{L^2}^2 + \snorm{\nabla^m w}{L^2}^2 \right).
\end{equation*}
\end{prop}

\begin{proof}
We have, for any $k$, $0 \leq k \leq m$
\begin{align*}
\dtime \snorm{\nabla^k w}{L^2}^2
		&= \dtime \int_M \bigl\langle \nabla^k w, \nabla^k w \bigl\rangle\, dv \\
		&= 2 \int_M \bigl\langle \dvar{t} \nabla^k w, \nabla^k w \bigl\rangle\, dv \\
		&= 2 \int_M \bigl\langle \hat{L} (\nabla^k w) + \nabla^k\left({Q}(u) - {Q}(v)\right),
													\nabla^k w \bigl\rangle\, dv \\
		&= 2 \int_M \bigl\langle \hat{L} (\nabla^k w), \nabla^k w \bigl\rangle\, dv
				+ 2 \int_M \bigl\langle \nabla^k\left({Q}(u) - {Q}(v)\right), \nabla^k w \bigl\rangle\, dv \\
		&= 2 \int_M \bigl\langle \hat{L} (\nabla^k w), \nabla^k w \bigl\rangle\, dv
				+ 2(-1)^k \int_M \bigl\langle \left({Q}(u) - {Q}(v)\right), \nabla^{2k} w \bigl\rangle\, dv \\
\end{align*}
We apply G{\aa}rding's inequality (Theorem \ref{Garding}) to the first integral above.  In particular since $-\hat{L}$ is strongly elliptic,
\begin{align*}
 2 \int_M \langle -\hat{L} (\nabla^k w), \nabla^k w \rangle\, dv
 	&\geq  c_k \snorm{\nabla^k w}{L^{2}_m}^2 - K_k \snorm{\nabla^k w}{L^2}^2 \\
	& \geq c_k \snorm{\nabla^m \nabla^k w}{L^2}^2 - K_k \snorm{\nabla^k w}{L^2}^2, \end{align*}
and so
\[ 2 \int_M \langle \hat{L} (\nabla^k w), \nabla^k w \rangle\, dv
	\leq  -c_k \snorm{\nabla^m \nabla^k w}{L^2}^2 + K_k \snorm{\nabla^k w}{L^2}^2. \]
For the second integral, we estimate using H\"older's inequality followed by Lemma \ref{quadest} and the interpolation inequalities.  A straightforward estimation shows that for any $\vep > 0$ there exists a $C > 0$ where
\[
\left|2(-1)^k \int_M \langle \left({Q}(u) - {Q}(v)\right), \nabla^{2k} w \rangle\, dv \right|
	\leq C\snorm{w}{L^2}^2 + \vep \snorm{\nabla^{2m}w}{L^2}^2.
\]

In total for $\vep > 0$ there exists a constant $C > 0$ where we obtain
\begin{equation*}
\dtime\snorm{\nabla^k w}{L^2}^2
	\leq - c_k \snorm{\nabla^{m+k} w}{L^2}^2 + K_k \snorm{\nabla^k w}{L^2}^2
			+ C \snorm{w}{L^2}^2 + \vep \snorm{\nabla^{2m} w}{L^2}^2.
\end{equation*}

We now combine this estimate with $k = 0$ and this estimate with $k = m$ to get
\begin{align*}
\dtime\Bigl(\snorm{w}{L^2}^2 + \snorm{\nabla^m w}{L^2}^2\Bigr)
		&\leq - c_0 \snorm{\nabla^m w}{L^2}^2 + K_0 \snorm{w}{L^2}^2
			+ C' \snorm{w}{L^2}^2 + \vep \snorm{\nabla^{2m} w}{L^2}^2\\
		& \quad - c_m \snorm{\nabla^{2m} w}{L^2}^2 + K_m \snorm{\nabla^m w}{L^2}^2
			+ C'' \snorm{w}{L^2}^2 + \vep \snorm{\nabla^{2m} w}{L^2}^2.
\end{align*}
Taking $\vep > 0$ so small that $-c_m + 2 \vep < 0$, we obtain the desired inequality, where $K = \max\{K_0+C' + C'', K_m\}$.
\end{proof}

Theorem \ref{thm:main-adj-flow} now follows immediately from this proposition.
\begin{proof}[Proof of Theorem \ref{thm:main-adj-flow}]
If $g_1$ and $g_2$ are two solutions to \eqref{originalflow-adj} then $w(0) = 0$ and so since $w$ must satisfy the differential inequality of Proposition \ref{differentialinequalityprop}, it must remain equal to zero.
\end{proof}


\section{Map Covariant Derivatives} \label{mcd}

We begin with general remarks concerning map covariant derivatives and then specialize to the setting appropriate for the geometric flows of interest.  While not new, the results here are meant to elucidate the machinery and nuance underlying map covariant derivatives.  For additional development, see \cite{ChowLN,EellsS,HamiltonSingularities}.

Let $\varphi: M \rightarrow N$ be a smooth map.  For any vector bundle $B$ over $N$, define the pullback bundle $\varphi^*(B)$ over $M$ by requiring the fiber over a point $p$ in $M$ to be $B_{\varphi(p)}$.  An important example is the pullback bundle of the tangent bundle $TN$.  With this, we can interpret the differential of $\varphi$ as a section of the bundle $T^*M \otimes \varphi^*TN$.  In coordinates, we have
\[
d\varphi = \frac{\partial \varphi^{\alpha}}{\partial x^i}
			dx^i \otimes \left.\frac{\partial}{\partial y^{\alpha}}\right|_{\varphi}
\]
where here and in the future, the vertical bar indicates composition and Greek indices are used for objects associated with bundles initially over $N$ while Latin indices are used for objects associated with bundles initially over $M$.

Let $M$ be equipped with a metric $\gbar$, and let $N$ be equipped with a metric $h$.  We consider the metric $h \circ \varphi$ on $\varphi^*TN$ and note that this is not $\varphi^*(h)$ which acts on elements of $TM$ instead of $\varphi^*TN$.  Using the Levi-Civita connections of these metrics, we may define a connection and associated covariant derivative called the map covariant derivative on sections of $(T^*M)^{p_1} \otimes (TM)^{q_1} \otimes (\varphi^* T^*N)^{p_2} \otimes (\varphi^* TN)^{q_2}$ which we call $\binom{p_1,p_2}{q_1,q_2}$-tensors, where we account for any part of the tensor coming from bundles originally over $M$ using $\gbar$ and any part of the tensor coming from bundles originally over $N$ using $h$ through composition with $\varphi$.  For example, the tensors of primary interest in this paper are $\binom{p,0}{0,q}$-tensors and for such a tensor $F$ we have
\[
(\nabla^{\gbar,h} F)^{\alpha_1 \cdots \alpha_q}_{i_1 \cdots i_p m}
		= \frac{\partial}{\partial x^m} (F^{\alpha_1 \cdots \alpha_q}_{i_1 \cdots i_p})
			- \sum_{l=1}^p F^{\alpha_1 \cdots \alpha_q}_{i_1 \cdots k \cdots i_p}
				(\Gamma^{\gbar})_{m i_l}^{k}
			+ \sum_{l=1}^q F^{\alpha_1 \cdots \gamma \cdots \alpha_q}_{i_1 \cdots i_p}
						\frac{\partial \varphi^{\mu}}{\partial x^{m}}
							(\Gamma^{h})_{\gamma \mu}^{\alpha_l} \circ \varphi.
\]
Finally, we extend the map covariant derivative to maps $f:M \rightarrow N$ by defining it to be the differential of $f$ and we note that as such, $\nabla^{\gbar,h} f = df$ is a $\binom{1,0}{0,1}$-tensor.

One consequence of the definition is that if a tensor is originally defined over just $M$ or just $N$, then the map covariant derivative of that tensor is just the covariant derivative with respect to the metric corresponding to the appropriate manifold.  As a specific example, this connection is compatible with both $\gbar$ and $h \circ \varphi$:

\begin{lemma}
The following identities hold:
\[
\nabla^{\gbar,h} \gbar = 0\ \ \ \mbox
{and}\ \ \ \ \nabla^{\gbar,h} (h \circ \varphi) = 0.
\]
\end{lemma}

Also, when commuting map covariant derivatives, identities involving curvature arise that are similar to those for the usual covariant derivative.  For our work here, we do not need the precise details but we note the following

\begin{prop} \label{mapcovswapprop}
For a tensor $F$ as given above,
\[
\nabla^{\gbar, h}_a \nabla^{\gbar, h}_b F^{\alpha_1 \cdots \alpha_q}_{i_1 \cdots i_p}
	- \nabla^{\gbar, h}_b \nabla^{\gbar, h}_a F^{\alpha_1 \cdots \alpha_q}_{i_1 \cdots i_p}
	= (F \star Rm^{\gbar} + F \star Rm^{h} \circ \varphi)^{\alpha_1 \cdots \alpha_q}_{i_1 \cdots i_p a b}.
\]
\end{prop}

Finally, we define the map Laplacian by
\[
\Delta^{\gbar, h} F = \tr^{\gbar} \left[ \nabla^{\gbar, h} \nabla^{\gbar, h} F  \right]
\]
where the trace is taken over the two indices corresponding to the map covariant derivatives.

\subsection{When the map is the identity}

Consider a special case of the above: suppose  $M = N$, and $\varphi = \id$.  In this setting, we have one manifold equipped with two metrics $g$ and $h$, and when taking a map covariant derivative it must be made clear which metric connection is acting on which part of the tensor.  Note that we do not use $\gbar$ here.  We will find that it is most natural to reserve this for the pullback of $g$ under another map.  For the application of interest, we are using $g$ for the contravariant part of the tensor and $h$ for the covariant part.  If we work in one coordinate patch $\{y^{\beta}\}$ then all the derivatives of the (identity) map simplify and we have

\[
(\nabla^{g,h} F)^{\alpha_1 \cdots \alpha_q}_{\beta_1 \cdots \beta_p \mu}
		= \frac{\partial}{\partial y^{\mu}} (F^{\alpha_1 \cdots \alpha_q}_{\beta_1 \cdots \beta_p})
			- \sum_{l=1}^p F^{\alpha_1 \cdots \alpha_q}_{\beta_1 \cdots \gamma \cdots \beta_p}
							(\Gamma^g)_{\mu \beta_l}^{\gamma}
			+ \sum_{l=1}^q F^{\alpha_1 \cdots \gamma \cdots \alpha_q}_{\beta_1 \cdots \beta_p}
							(\Gamma^h)_{\gamma \mu}^{\alpha_l}.
\]
Since there is no longer an explicit map involved, we call this the ``mixed covariant derivative.''  In this setting, we can ask how this mixed covariant derivative compares to the usual covariant derivative.  To answer this, we introduce the $\binom{2}{1}$-difference tensor
\[
A^{g,h} = \Gamma^g - \Gamma^h
\]
and we have the following 
\begin{prop} \label{puremixedcovprop}
Let $F$ be a $\binom{p}{q}$-tensor on $N$ and let $g$ and $h$ be two metrics on $N$.  Let $A^{g,h}$ be the difference tensor for $g$ and $h$.  Then
\[
(\nabla^{g} F)^{\alpha_1 \cdots \alpha_q}_{\beta_1 \cdots \beta_p \mu}
	= (\nabla^{g,h} F)^{\alpha_1 \cdots \alpha_q}_{\beta_1 \cdots \beta_p \mu}
		+ \sum_{l=1}^q F^{\alpha_1 \cdots \gamma \cdots \alpha_q}_{\beta_1 \cdots \beta_p}
				(A^{g,h})_{\gamma \mu}^{\alpha_l}
\]
or, masking the structure of the lower order terms, we have
\[
\nabla^{g} F = \nabla^{g,h} F + F \star A^{g,h}.
\]
\end{prop}

\begin{proof}
Compare the coordinate formulas for the two different covariant derivatives.
\end{proof}

As an example, we can apply Proposition \ref{puremixedcovprop} to the $\binom{1}{1}$-tensor $d\, \id$.

\begin{lemma} \label{mixedcovofdidlemma}
The following identity holds
\[
\nabla^{g,h} d\, \id = -A^{g,h}.
\]
\end{lemma}

\begin{proof}
This follows directly from Proposition \ref{puremixedcovprop} and the facts that $\nabla^{g} d\, \id = 0$ and $(d\, \id)^{\alpha}_{\beta} = \delta^{\alpha}_{\beta}$.
\end{proof}

We can generalize the previous result to account for higher derivatives.  In doing so, it will be useful to express all lower order derivatives in terms of the mixed covariant derivative:
\begin{prop} \label{mixedvspure-prop}
Given $F$ as above, taking $k$ derivatives, we have
\[
(\nabla^{g})^k F = (\nabla^{g,h})^k F
					+ \sP((\nabla^{g,h})^{k-1} F, (\nabla^{g,h})^{k-1} A^{g,h}).
\]
\end{prop}

\medskip

\subsection{Pulling back} \label{pullingbacksection}
Consider the composition of maps
\[
(M, \gbar) \overarrow{\varphi} (N, g) \overarrow{\id} (N,h)
\]
with $\gbar = \varphi^* g$.  For a $\binom{p}{q}$-tensor $F$ on $N$ define the pullback tensor $\varphi^*F$ as a $\binom{p,0}{0,q}$-tensor by
\[
\varphi^*F = (\varphi^*F)^{\alpha_1 \cdots \alpha_q}_{i_1 \cdots i_p}
			dx^{i_1} \otimes \cdots \otimes dx^{i_p} \otimes
		\left.\frac{\partial}{\partial y^{\alpha_1}}\right|_{\varphi}
			\otimes \cdots \otimes
				\left.\frac{\partial}{\partial y^{\alpha_q}}\right|_{\varphi}
\]
where
\[
(\varphi^*F)^{\alpha_1 \cdots \alpha_q}_{i_1 \cdots i_p}
	= (F^{\alpha_1 \cdots \alpha_q}_{\beta_1 \cdots \beta_p} \circ \varphi)
		\frac{\partial \varphi^{\beta_1}}{\partial x^{i_1}}
			\cdots \frac{\partial \varphi^{\beta_p}}{\partial x^{i_p}}.
\]
To summarize we are pulling the covariant part of $F$ back by the usual pullback rules, but for the contravariant part we are simply composing with $\varphi$.

As a specific example that proves useful later, since ${d\, \id }$ is a $\binom{1}{1}$-tensor, $\varphi^*(d \, \id)$ is a $\binom{1,0}{0,1}$-tensor, and we have the following:
\begin{lemma} \label{pullbackofdidlemma}
The following identity holds:
\[
\varphi^*(d \, \id) = d \varphi.
\]
\end{lemma}
\begin{proof}
This is a direct computation:
\begin{align*}
[\varphi^*(d \, \id)]_i^{\alpha} &= \bigl((d\, \id)_{\beta}^{\alpha} \circ \varphi \bigr)
						\frac{\partial \varphi^{\beta}}{\partial x^i} \\
					&= (\delta_{\beta}^{\alpha} \circ \varphi)
						\frac{\partial \varphi^{\beta}}{\partial x^i} \\
					&= \frac{\partial \varphi^{\alpha}}{\partial x^i} \\
					&= (d \varphi)_i^{\alpha}.
\end{align*}
\end{proof}

Note that, with the definition of pullback given, $(\gbar)^{-1} = (\varphi^*g)^{-1}$ is a $\binom{0,0}{2,0}$-tensor on $M$, while $\varphi^*(g^{-1}) = (g^{-1}) \circ \varphi$ is a $\binom{0,0}{0,2}$-tensor on $M$ and so they are not equal to each other.  As such, when performing a metric contraction of tensors, some care must be taken to ensure that the appropriate metric is being used for the task.  The following proposition shows that contraction, appropriately interpreted, is natural with respect to this pullback.
\begin{prop} \label{pullbacktrace-prop}
Let $F$ be a $\binom{p}{q}$-tensor on $N$, and let $\varphi^*F$ be the pullback tensor on $M$ defined as above.  Then contracting on any pair of lower indices
\[
\tr^{\gbar} \varphi^*F = \varphi^*(\tr^g F)
\]
and contracting on any pair of upper indices
\[
\tr^{h \circ \varphi} \varphi^*F = \varphi^*(\tr^h F)
\]
\end{prop}

\begin{proof}
These are just computations, but it is worth noting how the contraction is transferred between $M$ and $N$.  For the first identity, letting $p = r+2$ and contracting without loss of generality on the first two indices, we have
\begin{align*}
(\tr^{\gbar} \varphi^*F)^{\alpha_1 \cdots \alpha_q}_{i_1 \cdots i_r}
		&= \gbar^{jk} (\varphi^*F)^{\alpha_1 \cdots \alpha_q}_{j k i_1 \cdots i_r} \\
		&= \gbar^{jk} (F^{\alpha_1 \cdots \alpha_q}_{\mu \nu \beta_1 \cdots \beta_r} \circ \varphi)
					\frac{\partial \varphi^{\mu}}{\partial x^j}
					\frac{\partial \varphi^{\nu}}{\partial x^k} 
					\frac{\partial \varphi^{\beta_1}}{\partial x^{i_1}}
						\cdots \frac{\partial \varphi^{\beta_r}}{\partial x^{i_r}} \\
		&= [(g^{\mu \nu}
				F^{\alpha_1 \cdots \alpha_q}_{\mu \nu \beta_1 \cdots \beta_r}) \circ \varphi]
					\frac{\partial \varphi^{\beta_1}}{\partial x^{i_1}}
						\cdots \frac{\partial \varphi^{\beta_r}}{\partial x^{i_r}} \\
		&= [(\tr^g F)^{\alpha_1 \cdots \alpha_q}_{\beta_1 \cdots \beta_r} \circ \varphi]
					\frac{\partial \varphi^{\beta_1}}{\partial x^{i_1}}
						\cdots \frac{\partial \varphi^{\beta_r}}{\partial x^{i_r}} \\
		&= \varphi^*(\tr^g F)^{\alpha_1 \cdots \alpha_q}_{i_1 \cdots i_r}.
\end{align*}

The second identity is similar but simpler because the metric performing the contraction is only adjusted by a composition.
\end{proof}

We may also compute the map covariant derivative of the tensor $\varphi^*F$ and compare it to the mixed covariant derivative of $F$.  The following result shows that the definition for the pullback given above ensures the naturality of the map covariant derivative.
\begin{prop} \label{pullbackcov-prop}
Let $F$ be a $\binom{p}{q}$-tensor on $N$, and let $\varphi^*F$ be the pullback tensor defined as above.  Then
\[
(\nabla^{\gbar,h} \varphi^*F)^{\alpha_1 \cdots \alpha_q}_{i_1 \cdots i_p m}
		= \bigl((\nabla^{g, h} F)^{\alpha_1 \cdots \alpha_q}_{\beta_1 \cdots \beta_p \mu}
				\circ \varphi \bigr)
			\frac{\partial \varphi^{\beta_1}}{\partial x^{i_1}}
				\cdots \frac{\partial \varphi^{\beta_p}}{\partial x^{i_p}}
					\frac{\partial \varphi^{\mu}}{\partial x^{m}}
\]
or without indices
\[
\nabla^{\gbar, h} (\varphi^* F) = \nabla^{\varphi^*(g), h} (\varphi^* F) = \varphi^* (\nabla^{g, h} F)
\]
\end{prop}

\begin{proof}
This is analogous to the proof of naturality for a the Levi-Civita connection.
\end{proof}

An immediate consequence of the previous propositions is that the Laplacian pulls back in the expected way:

\begin{prop} \label{laplacianpullbackprop}
Let $F$ be a $\binom{p}{q}$-tensor on $N$, and let $\varphi^*F$ be the pullback tensor defined as above.  Then
\[
\Delta^{\gbar,h} \varphi^*F = \varphi^*(\Delta^{g,h} F)
\]
\end{prop}

\begin{proof}
This follows immediately from Propositions \ref{pullbacktrace-prop} and \ref{pullbackcov-prop}.
\end{proof}

To finish this section, we apply the preceding results in some examples that will be used in the next section:  Using the difference tensor for $g$ and $h$, define the vector fields
\begin{align*}
(V^{g,h})^{\gamma} &= g^{\alpha \beta} (A^{g,h})_{\alpha \beta}^{\gamma} \\
(Z^{g,h})^{\gamma} &= g^{\mu \alpha} g^{\nu \beta}
				\nabla^{g,h}_{\mu} \nabla^{g,h}_{\nu} (A^{g,h})_{\alpha \beta}^{\gamma}.
\end{align*}
We will make use of powers of the Laplacian of these vector fields and we will need to know how they pull back.  For this, note that $(\Delta^{g,h})^k V^{g,h}$ and $(\Delta^{g,h})^{k-1} Z^{g,h}$, where $k$ indicates the power, are $\binom{0}{1}$-tensors and so their pullbacks will be $\binom{0,0}{0,1}$-tensors.  Specifically, we have the following:
\begin{prop} \label{vectorfieldpullback-prop}
The following relations hold:
\begin{align*}
\varphi^*(A^{g,h}) &= -\nabla^{\gbar,h} d \varphi, \\
\varphi^*\left[(\Delta^{g,h})^{k}V^{g,h}\right] &= -(\Delta^{\gbar,h})^{k+1} \varphi, \\
\varphi^*\left[(\Delta^{g,h})^{k-1}Z^{g,h}\right]
		&= -(\Delta^{\gbar,h})^{k+1} \varphi
				+ \P(\partial^{2k+1} \gbar, (\partial^{2k+1} h) \circ \varphi, \partial^{2k} \varphi).
\end{align*}
\end{prop}

\begin{proof}
For the first identity, applying Lemma \ref{mixedcovofdidlemma}, Proposition \ref{pullbackcov-prop}, and then Lemma \ref{pullbackofdidlemma} to the difference tensor we have
\begin{align*}
\varphi^*(A^{g,h}) &= \varphi^*(-\nabla^{g,h} d\, \id) \\
			&= -\nabla^{\gbar, h} \varphi^*(d\, \id) \\
			&= -\nabla^{\gbar, h} d \varphi.
\end{align*}

The second identity is proved by induction.  For the base case we use Proposition \ref{pullbacktrace-prop} and the first identity to get
\begin{align*}
\varphi^*(V^{g,h}) &= \varphi^*(\tr^g A^{g,h}) \\
		&= \tr^{\gbar} (\varphi^* A^{g,h}) \\
		&= \tr^{\gbar}(-\nabla^{\gbar, h} d \varphi) \\
		&= - \Delta^{\gbar, h} d \varphi.
\end{align*}
The induction step is a direct application of Proposition \ref{laplacianpullbackprop}.

The third identity also follows by induction.  The base case is similar to that above, but we take some care to track which indices are being contracted in order to make correct use of Proposition \ref{mapcovswapprop}.  We have
\begin{align*}
\bigl[\varphi^*(Z^{g,h})\bigr]^{\gamma} &= \varphi^*\bigl[g^{\mu \alpha} g^{\nu \beta}
					\nabla^{g,h}_{\mu} \nabla^{g,h}_{\nu}
					(A^{g,h})_{\alpha \beta}^{\gamma}\bigr] \\
			&= \gbar^{ia} \gbar^{jb} (\nabla^{\gbar, h}_i \nabla^{\gbar, h}_j
					\nabla^{\gbar, h}_a \nabla^{\gbar, h}_b \varphi)^{\gamma} \\
			&= \gbar^{ia} \gbar^{jb} (\nabla^{\gbar, h}_i \nabla^{\gbar, h}_a						\nabla^{\gbar, h}_j \nabla^{\gbar, h}_b \varphi)^{\gamma}
				+ \Bigl(\nabla^{\gbar, h}_i
					\bigl(\nabla^{\gbar, h}_b \varphi \star Rm^{\gbar}
						+ \nabla^{\gbar, h}_b \varphi
							\star Rm^h \circ \varphi \bigr)_{aj}^{\gamma}\Bigr) \\
			&= \bigl((\Delta^{\gbar, h})^2 \varphi \bigr)^{\gamma}
				+ \P(\partial^{3} \gbar, (\partial^{3} h) \circ \varphi, \partial^2 \varphi).
\end{align*}
Again, the induction step is a direct application of Proposition \ref{laplacianpullbackprop}.
\end{proof}


\section{Uniqueness for the Pure Flow} \label{uniqueness:pure-flow}

In this section, we prove the uniqueness part of Theorem \ref{mainuniquetheorem}.  The existence part of Theorem \ref{mainuniquetheorem} was proved in \cite{BahuaudHelliwell}.  We are working with polynomial natural $\binom{2}{0}$-tensors where the top order terms are of the form given in equation \eqref{ansatz}, which we restate here for convenience
\begin{equation}
\tag{4}
T^g_{k,a,b,c} = (-1)^{k+1} c\bigl(\Delta^{k} Ric
			+ a \Delta^{k} S g - b \Delta^{k-1} \nabla^2 S\bigr) + T_{2k+1}.
\end{equation}
We note that here and in the rest of this section, all geometric objects not adorned by a metric are associated with the metric $g$.  The ambient obstruction tensor is of the form \eqref{ansatz}. In fact, application of the Bianchi identities and the fact that covariant derivatives commute at top order imply that these three top order terms are all that can appear when considering covariant two-tensors resulting from contractions of derivatives of curvature of a single metric.  Moreover, we will see in the symbol analysis that in order to have a strongly elliptic operator, there must be some contribution from $\Delta^k Ric$.  

Due to the diffeomorphism invariance of \eqref{ansatz} we modify the flow by the Lie derivative of the metric with respect to a vector field to obtain a strictly parabolic operator.  We select an appropriate vector field built from contractions of the $\binom{2k+2}{1}$ tensor $(\nabla^{g,h})^{2k} A^{g,h}$.  The choice to use only mixed covariant derivatives instead of covariant derivatives with respect to $g$ is simply to make the map flow analysis simpler, and Proposition \ref{mixedvspure-prop} guarantees that this can be done without adding any complication.

Interestingly there are number of possible vector fields that can serve as building blocks, all coming from different choices of contractions of the indices.  Natural contractions (those not requiring a metric to raise or lower an index), when pulled back by a diffeomorphism, involve the inverse of the diffeomorphism, thus complicating the analysis of the resulting operator.  Authors using vector fields with natural contractions neither avoid this issue nor address how to resolve it.  The choices we make simplify the subsequent analysis by avoiding natural contractions.  One of the vector fields we use, $V^{g,h}$, is a generalization of that used for Ricci flow and is widely used.  The other vector field, $Z^{g,h}$, is (modulo lower order terms) the only other vector field built from $(\nabla^{g,h})^{2k} A^{g,h}$ that avoids natural contractions and is a new alternative to other choices.

The proof of the uniqueness part of Theorem \ref{mainuniquetheorem} proceeds as follows:  First, we perform the necessary symbol analysis so that we may apply Theorem \ref{thm:main-adj-flow} to an appropriate adjusted flow.  Second, we show that the vector field chosen for the adjustment gives rise to an elliptic operator on maps between manifolds, ensuring the existence of such maps.  Finally, we show that solutions to the pure flow give rise to solutions for the adjusted flow and that uniqueness for the adjusted flow passes to uniqueness for the pure flow via these maps.

\subsection{Choosing the vector field and symbol analysis for the adjusted flow}
We recall the two vector fields defined using the difference tensor introduced in Section \ref{pullingbacksection}:
\begin{align*}
 V^{g,h} &=  g^{\alpha \beta} (A^{g,h})_{\alpha \beta}^{\gamma} \partial_{\gamma}, \; \mbox{and} \\
 Z^{g,h} &=  g^{\mu \alpha} g^{\nu \beta} \nabla^{g,h}_{\mu} \nabla^{g,h}_{\nu} (A^{g,h})_{\alpha \beta}^{\gamma} \partial_{\gamma}.
\end{align*}

The following lemma computes the principal symbols of various terms we need.  

\begin{lemma} \label{symbolcomplemma}
The principal symbols of the linearizations of the basic building blocks in \eqref{ansatz} and of the Lie derivatives of the metric with respect to the vector fields we will use are:
\begin{align*}
\langle \sigma(\Delta^k Ric), \eta \rangle
		&= (-1)^{k} \frac{1}{2} |\xi|^{2k-2} \Bigl(|\xi|^4 |\eta|^2
			+ |\xi|^2 \tr\, \eta \langle \xi \otimes \xi, \eta \rangle
			- 2|\xi|^2 \langle \xi \otimes \xi, \tr(\eta \otimes \eta) \rangle \Bigr) \\
\langle \sigma(\Delta^k S g), \eta \rangle
		&= (-1)^{k} |\xi|^{2k-2} \left[\bigl(|\xi|^2 \tr\, \eta \bigr)^2
			- |\xi|^2 \tr\, \eta \langle \xi \otimes \xi, \eta \rangle \right] \\
\langle \sigma(\Delta^{k-1} \nabla^2 S), \eta \rangle
		&= (-1)^{k-1} |\xi|^{2k-2} \Bigl(\langle \xi \otimes \xi, \eta \rangle^2
			- |\xi|^2 \tr\, \eta \langle \xi \otimes \xi, \eta \rangle \Bigr)
\end{align*}
and
\begin{align*}
\langle \sigma \left(\Lie{(\Delta^{g,h})^k V^{g,h}}{g} \right), \eta \rangle
		&= (-1)^{k-1}|\xi|^{2k-2}\Bigl(2|\xi|^2\langle \xi \otimes \xi, \tr(\eta \otimes \eta) \rangle
			- |\xi|^2 \tr\, \eta \langle \xi \otimes \xi, \eta \rangle \Bigr) \\
\langle \sigma \left(\Lie{(\Delta^{g,h})^{k-1} Z^{g,h}}{g} \right), \eta \rangle
		&= (-1)^{k-1} |\xi|^{2k-2} \Bigl(2|\xi|^2\langle \xi \otimes \xi, \tr(\eta \otimes \eta) \rangle
			- \langle \xi \otimes \xi, \eta \rangle^2 \Bigr)
\end{align*}
where $\tr(\eta \otimes \eta)_{jk} = g^{il} \eta_{ij} \eta_{kl}$.  
\end{lemma}

For parameters $a, b, c, \alpha, \beta$ to be chosen, we now consider the parabolicity of the adjusted flow
\[
\dvar{t} g = T^g_{k,a,b,c} + \Lie{W^{g,h}_{c,\alpha, \beta}}{g} = L^{g,h}_{k,a,b,c,\alpha,\beta}
\]
where
\begin{align*}
T^g_{k,a,b,c} &= (-1)^{k+1} c\bigl(\Delta^{k} Ric
			+ a \Delta^{k} S g - b \Delta^{k-1} \nabla^2 S\bigr) + T_{2k+1}, \; \mbox{and} \\
W^{g,h}_{c,\alpha,\beta} &= (-1)^{k} c \bigl[\alpha (\Delta^{g,h})^k V^{g,h}
						+ \beta (\Delta^{g,h})^{k-1} Z^{g,h} \bigr].
\end{align*}
In view of Lemma \ref{symbolcomplemma}, the principal symbol of $-L^{g,h}_{k,a,b,c,\alpha,\beta}$ is seen to be
\begin{align*}
\langle \sigma (-L^{g,h}_{k,a,b,c,\alpha,\beta}), \eta \rangle
	= c |\xi|^{2k-2}\biggl[&\frac{1}{2}|\xi|^4 |\eta|^2 \\
	& \quad + (2\alpha + 2\beta - 1) |\xi|^2 \langle \xi \otimes \xi, \tr(\eta \otimes \eta) \rangle \\
	& \quad + (b - \beta) \langle \xi \otimes \xi, \eta \rangle^2 \\
	& \quad + \left(\frac{1}{2} - a - b - \alpha \right) |\xi|^2 \tr\, \eta \langle \xi \otimes \xi, \eta \rangle \\
	& \quad + a (|\xi|^2 \tr\, \eta)^2 \biggr].
\end{align*}

Note that the $|\xi|^4 |\eta|^2$ term arises only from $\Delta^{k} Ric$.  It can be shown by analysis similar to that below that for $-L^{g,h}_{k,a,b,c,\alpha,\beta}$ to be strongly elliptic, this contribution is necessary and that it must be positive.  Once this is known, the tensors of interest can indeed be written in the form \eqref{ansatz}.  Then we choose the coefficients $\alpha$ and $\beta$ in a manner to simultaneously force the vanishing of the second term, and produce a particular choice of coefficient for the fourth term that will lead to a perfect square in the subsequent analysis.  In particular we have

\begin{prop}
Let $c > 0$ and set $\alpha = \frac{1}{2} + a - b$ and $\beta = b - a$.  If $a > -\frac{1}{2(n-1)}$ then $-L^{g,h}_{k,a,b,c,\frac{1}{2} + a - b,b - a}$ is strongly elliptic.  If $a \leq -\frac{1}{2(n-1)}$, then $-L^{g,h}_{k,a,b,c,\alpha,\beta}$ is not strongly elliptic for any choice of $\alpha$ and $\beta$.
\end{prop}
\begin{proof}
Define $L^{g,h}_{k,a,b,c} = L^{g,h}_{k,a,b,c,\frac{1}{2} + a - b,b - a}$ and observe
\begin{align} \label{local-eqn-parabolicity}
\langle \sigma (-L^{g,h}_{k,a,b,c}), \eta \rangle
	&= c |\xi|^{2k-2} \biggl[ \frac{1}{2} |\xi|^4 |\eta|^2
				+ a \bigl(\langle \xi \otimes \xi, \eta \rangle^2
				- 2 |\xi|^2 \tr\, \eta \langle \xi \otimes \xi, \eta \rangle
				+ (|\xi|^2 \tr\, \eta)^2 \bigr) \biggr] \nonumber \\
	&= c |\xi|^{2k-2} \biggl[ \frac{1}{2} |\xi|^4 |\eta|^2
				+ a \bigl( \langle \xi \otimes \xi, \eta \rangle - |\xi|^2 \tr\, \eta \bigr)^2 \biggr] \nonumber  \\
	&= c |\xi|^{2k-2} \biggl[ \frac{1}{2} |\xi|^4 |\eta|^2
				+ a \langle \xi \otimes \xi - |\xi|^2 g, \eta \rangle^2 \biggr],
\end{align}
where we use that $|\xi|^2 \tr\, \eta = \langle |\xi|^2 g, \eta \rangle$ to obtain the final inequality.

To check strong ellipticity, we need to confirm that there exists a positive constant $\Lambda > 0$ where for all $\xi$ and $\eta$, 
\[
\bigl\langle \sigma(-L^{g,h}_{k,a,b,c}),\eta \bigr\rangle^h \geq  \Lambda (|\xi|^h)^{2k+2} (|\eta|^h)^2.
\]
Since we are working on a compact manifold, the norms with respect to $h$ are equivalent to those of $g$.  Moreover, since $g(0) = h$, the inner product with respect to $g$ and the inner product with respect to $h$ are comparable for a short time.  Therefore it is enough to show that
\[
\bigl\langle \sigma(-L^{g,h}_{k,a,b,c}),\eta \bigr\rangle \geq  \Lambda |\xi|^{2k+2} |\eta|^2.
\]
Looking at the equation \eqref{local-eqn-parabolicity} above, we find that ellipticity is automatic when $a \geq 0$ since the second term on the right hand side is a square.  For $a < 0$, we use the Cauchy Schwarz inequality to get
\begin{align*}
\langle  \xi \otimes \xi - |\xi|^2g, \eta \rangle^2
		&\leq \bigl| \xi \otimes \xi - |\xi|^2g\bigr|^2 |\eta|^2 \\
		&= (n-1)|\xi|^4 |\eta|^2,
\end{align*}
with equality when $\eta = \xi \otimes \xi - |\xi|^2g$.  Thus we have
\begin{align*}
\langle \sigma (-L^{g,h}_{k,a,b,c}), \eta \rangle
	&\geq c |\xi|^{2k-2} \biggl[ \frac{1}{2} |\xi|^4 |\eta|^2
				+ a (n-1)|\xi|^4 |\eta|^2 \biggr] \\
	&= c \left(\frac{1}{2} + a(n-1)\right)|\xi|^{2k+2} |\eta|^2.
\end{align*}
Therefore we have strong ellipticity for all $a > -\frac{1}{2(n-1)}$.

Next, consider the case where $a \leq -\frac{1}{2(n-1)}$.  We find that when $\eta = |\xi|^2 g - \xi \otimes \xi$, we have
\begin{align*}
&\langle \xi \otimes \xi, \tr(\eta \otimes \eta) \rangle = 0 \\
&\langle \xi \otimes \xi, \eta \rangle = 0 \\
&\tr\, \eta = (n-1)|\xi|^2 \\
&|\eta|^2 = (n-1) |\xi|^4 = |\xi|^2 \tr\, \eta
\end{align*}
so that, regardless of our choices for $\alpha$ and $\beta$,
\begin{align*}
\langle -L^{g,h}_{k,a,b,c,\alpha,\beta}, \eta \rangle
	&= c|\xi|^{2k-2} \biggl[\frac{1}{2}|\xi|^4 |\eta|^2 + a(|\xi|^2 \tr\, \eta)^2 \biggr] \\
	&= c|\xi|^{2k-2} \biggl[\frac{1}{2}|\xi|^4 |\eta|^2 + a(n-1)|\xi|^4|\eta|^2 \biggr] \\
	&= c\left(\frac{1}{2}+a(n-1)\right)|\xi|^{2k-2} |\eta|^2
\end{align*}
and so the operator is not elliptic when $a \leq -\frac{1}{2(n-1)}$.
\end{proof}

As as a consequence of this proposition and Theorem \ref{thm:main-adj-flow}, the adjusted flow
\begin{equation*}
\left\{ \begin{aligned} 
\dvar{t} g &= T^g_{k,a,b,c} + \Lie{W^{g,h}_{c,\frac{1}{2}+a-b,b-a}}{g}\\
g(0) &= h
\end{aligned} \right.
\end{equation*}
 has a unique short-time solution if $c > 0$ and $a > \frac{-1}{2(n-1)}$.

\subsection{Parabolicity for the diffeomorphism}

We now turn our attention to showing that the vector field $W^{g,h}_{c,\frac{1}{2} + a - b,b-a}$ used in adjusting the flow can be expressed as an elliptic operator on a map between manifolds.  Using Proposition \ref{vectorfieldpullback-prop} we compute

\begin{align*}
W^{g,h}_{c,\frac{1}{2} + a - b,b-a} \circ \varphi
		&= (-1)^{k} c \left[\left(\frac{1}{2} + a - b\right) (\Delta^{g,h})^k V^{g,h} \circ \varphi
						+ (b - a) (\Delta^{g,h})^{k-1} Z^{g,h} \circ \varphi \right] \\
		&= (-1)^{k} c \left[\left(\frac{1}{2} + a - b\right) \varphi^*(\Delta^{g,h})^k V^{g,h}
						+ (b - a) \varphi^*(\Delta^{g,h})^{k-1} Z^{g,h} \right] \\
		&= (-1)^{k} c \left[\left(\frac{1}{2} + a - b\right) (-(\Delta^{\gbar,h})^{k+1} \varphi)
						+ (b - a) (-(\Delta^{\gbar,h})^{k+1} \varphi) \right] \\
		&\qquad
			+ \P(\partial^{2k+1} \gbar, (\partial^{2k+1} h) \circ \varphi, \partial^{2k} \varphi) \\
		&= (-1)^{k+1} \frac{c}{2} (\Delta^{\gbar,h})^{k+1} \varphi
			+ \P(\partial^{2k+1} \gbar, (\partial^{2k+1} h) \circ \varphi, \partial^{2k} \varphi).
\end{align*}
With this, we see that
\[
\dvar{t} \varphi = -W^{g,h}_{c,\frac{1}{2} + a - b,b-a} \circ \varphi
\]
becomes
\[
\dvar{t} \varphi = (-1)^{k} \frac{c}{2} (\Delta^{\gbar,h})^{k+1} \varphi
				+ \P(\partial^{2k+1} \gbar, (\partial^{2k+1} h) \circ \varphi, \partial^{2k} \varphi).
\]
When paired with the initial condition $\varphi(0) = \id$, this is solvable by parabolic theory on a compact manifold.

For future reference, let $E^{\gbar,h}_{a,b,c}$ be the operator
\[
E^{\gbar,h}_{a,b,c} (\varphi)
		= (-1)^{k} \frac{c}{2} (\Delta^{\gbar,h})^{k+1} \varphi
			+ \P(\partial^{2k+1} \gbar, (\partial^{2k+1} h) \circ \varphi, \partial^{2k} \varphi)
\]
corresponding to the vector field $W^{g,h}_{c,\frac{1}{2} + a - b,b-a}$.

\subsection{Uniqueness for the Pure Flow}
The final step is to to show that a solution to the pure flow transfers into a solution of the adjusted flow and then use uniqueness for the adjusted flow to establish uniqueness for the pure flow.  

A solution to the pure flow transfers to a solution of the adjusted flow whenever the tensor defining the pure flow is natural and the diffeomorphisms correspond to the vector field defining the adjustment:

\begin{lemma} \label{transitionlemma}
Let $T^{\gbar}$ be a smooth natural tensor for a single metric $\gbar$ and suppose $\gbar$ solves
\begin{equation*}
\left\{ \begin{aligned} 
\dvar{t} \gbar &= T^{\gbar} \\
\gbar(0) &= h.
\end{aligned} \right.
\end{equation*}
Let $\varphi: I \times (M,\gbar) \rightarrow (M,h)$ be a time-dependent diffeomorphism solving
\begin{equation*}
\left\{ \begin{aligned} 
\dvar{t} \varphi &= E^{\gbar,h} (\varphi) \\
\varphi(0) &= \id.
\end{aligned} \right.
\end{equation*}
Let $g(t) = [\varphi(t)]_* \gbar(t) = [\varphi^{-1}(t)]^* \gbar(t)$ and let $-W^{g,h}$ be the time-dependent vector field associated to the flow coming from $\varphi$ so that
\[
E^{\gbar,h} (\varphi) = -W^{g,h} \circ \varphi.
\]
Then $g$ solves the adjusted flow
\begin{equation*}
\left\{ \begin{aligned} 
\dvar{t} g &= T^g + \Lie{W^{g,h}}{g} \\
g(0) &= h.
\end{aligned} \right.
\end{equation*}
\end{lemma}

The details for the proof of this result are now standard.  See for example \cite{ChowK} for a proof in the context of Ricci flow.

Now let $\gbar_1$ and $\gbar_2$ both solve the pure flow
\begin{equation*}
\left\{ \begin{aligned} 
\dvar{t} \gbar &= T^{\gbar}_{k,a,b,c} \\
\gbar(0) &= h
\end{aligned} \right.
\end{equation*}
where $c>0$ and $a > -\frac{1}{2(n-1)}$.  For each $\gbar_i$ let $\varphi_i$ be a diffeomorphism that solves
\begin{equation*}
\left\{ \begin{aligned} 
\dvar{t} \varphi_i &= E^{\gbar_i,h}_{a,b,c}(\varphi_i) \\
\varphi_i(0) &= \id.
\end{aligned} \right.
\end{equation*}
By Lemma \ref{transitionlemma}, these diffeomorphisms give rise to two solutions $g_i(t) = [\varphi_i^{-1}(t)]^* \gbar_i(t)$ of the adjusted flow
\begin{equation*}
\left\{ \begin{aligned} 
\dvar{t} g &= T^g_{k,a,b,c} + \Lie{W^{g,h}_{c,\frac{1}{2} + a - b,b-a}}{g}\\
g(0) &= h.
\end{aligned} \right.
\end{equation*}
By Theorem \ref{thm:main-adj-flow}, $g_1 = g_2$ and as a consequence, the adjusting vector fields are also the same.  This implies that $\varphi_1$ and $\varphi_2$ both solve
\begin{equation*}
\left\{ \begin{aligned} 
\dvar{t} \varphi &= -W^{g,h}_{c,\frac{1}{2} + a - b,b-a} \circ \varphi\\
\varphi &= \id
\end{aligned} \right.
\end{equation*}
but, interpreted as an ODE at each point of $M$, solutions must be unique and we conclude $\varphi_1 = \varphi_2$.  Finally this means that
\[
\gbar_1(t) = [\varphi_1(t)]^*g_1(t) = [\varphi_2(t)]^* g_2(t) = \gbar_2(t)
\]
which establishes uniqueness and finishes the proof of Theorem \ref{mainuniquetheorem}.


\bibliographystyle{amsalpha}
\bibliography{geomflow2}

\end{document}